\documentclass[reqno,11pt]{amsart}
\usepackage{amsmath,amsfonts,amssymb,amsthm,graphicx,a4wide,enumerate,url}
\usepackage{color}
\usepackage[small,bf]{caption} 
\newtheorem{theorem}{Theorem}

\theoremstyle{definition}
\newtheorem{definition}{Definition}

\newcommand{\prn}[1]{\left(#1\right)}
\newcommand{\vecpower}[2]{\vec{#1}^{\,#2}}

\newcommand{\veci}[1]{\vec{#1}}					
\newcommand{\vecipower}[2]{\veci{#1}^{\,#2}}			
\newcommand{\vecc}{{\veci{c}}} 					
\newcommand{\vece}{\veci{e}} 					
\newcommand{\dt}{{\Delta t}} 					
\newcommand{\err}{\epsilon}					
\newcommand{\lte}{\delta}					
\newcommand{\sov}[1]{\veci{\tau}^{\,(#1)}} 			
\newcommand{\soq}{\tilde{q}}					

\newcommand{\sm}[1]{\scalebox{.75}{#1}}

\begin{document}
\parskip.9ex

\title[DIRK Schemes with High Weak Stage Order]
{DIRK Schemes with High Weak Stage Order}

\author[D. Ketcheson]{David Ketcheson}
\address[David Ketcheson]
{Applied Mathematics and Computational Science \\ King Abdullah University of Science and Technology \\
4700 Thuwal 23955-6900 \\ Saudi Arabia}
\email{david.ketcheson@kaust.edu.sa}
\urladdr{http://www.davidketcheson.info}

\author[B. Seibold]{Benjamin Seibold}
\address[Benjamin Seibold]
{Department of Mathematics \\ Temple University \\
1805 North Broad Street \\ Philadelphia, PA 19122}
\email{seibold@temple.edu}
\urladdr{http://www.math.temple.edu/\~{}seibold}

\author[D. Shirokoff]{David Shirokoff}
\address[David Shirokoff]
{Department of Mathematical Sciences \\	New Jersey Institute of Technology \\
University Heights \\ Newark, NJ 07102}
\email{david.g.shirokoff@njit.edu}
\urladdr{https://web.njit.edu/\~{}shirokof}

\author[D. Zhou]{Dong Zhou}
\address[Dong Zhou]
{Department of Mathematics \\ California State University Los Angeles \\
1805 North Broad Street \\ Philadelphia, PA 19122}
\email{dzhou11@calstatela.edu}



\begin{abstract}
Runge-Kutta time-stepping methods in general suffer from order reduction: the observed order of convergence may be less than the formal order when applied to certain stiff problems. Order reduction can be avoided by using methods with high stage order. However, diagonally-implicit Runge-Kutta (DIRK) schemes are limited to low stage order. In this paper we explore a weak stage order criterion, which for initial boundary value problems also serves to avoid order reduction, and which is compatible with a DIRK structure. We provide specific DIRK schemes of weak stage order up to 3, and demonstrate their performance in various examples.
\end{abstract}

\maketitle

\section{Introduction}
\label{sec:intro}
Runge-Kutta (RK) methods achieve high-order accuracy in time by means of combining approximations to the solution at multiple stages. An $s$-stage RK scheme can be represented via the Butcher tableau
\begin{equation*}
	\begin{array}{c|c}
		\vecc & A\ \\ \hline
		& \vecipower{b}{T}
	\end{array}
	\,=\,
	\begin{array}{c|cccc}
		c_1 	& a_{11} &  \cdots & a_{1s}\\
		\vdots	& \vdots &         & \vdots\\
		c_s		& a_{s1} &  \cdots & a_{ss}\\ \hline
		& b_{1}  &  \cdots & b_{s}
	\end{array}\;.
\end{equation*}
Throughout the whole paper we assume that $\vecc=A\vece$, where $\vece$ is the vector of all ones. The scheme's stability function \cite{WannerHairer1991} $R(\zeta) = 1+\zeta\vecipower{b}{T}(I - \zeta A)^{-1}\vece$ measures the growth $u^{n+1}/u^n$ per step $\dt$, when applying the scheme to the linear model equation $u'(t) = \lambda u$, with $\zeta = \lambda\dt$.

A particular interest lies in the accuracy of the RK scheme for stiff problems, i.e., problems in which a larger time step is chosen than the fastest time scale of the problem's dynamics. A standard stiff model problem \cite{ProtheroRobinson1974} is the scalar linear ordinary differential equation (ODE)
\begin{equation}
\label{eq:linear_ode_test_problem}
u' = \lambda (u-\phi(t)) + \phi'(t)\;,
\end{equation}
with i.c.\ $u(0) = \phi(0)$ and $\mbox{Re}\,\lambda\leq 0$. The true solution $y(t) = \phi(t)$ evolves on an $O(1)$ time scale. Hence, $\lambda$-values with large negative real part result in stiffness. Considering a family of test problems (parametrized by $\lambda$), one can now establish the scheme's convergence via two different limits: (a)~the non-stiff limit $\Delta t\to 0$ and $\zeta \to 0$; and (b)~the stiff limit $\Delta t\to 0$ and $\zeta \to -\infty$. A characteristic property of most RK schemes is that, while the non-stiff limit recovers the scheme's order (as given by the order conditions \cite{ButcherBook2008,HairerNorsettWanner1993}), the error decays at a reduced order in the stiff limit. This phenomenon is called ``order reduction'' (OR)
\cite{Verwer1986,SanzVerwerHundsdorfer1986,OstermannRoche1992,CarpenterGottliebAbarbanelDon1995,BurragePetzold1990} and it manifests in various ways for more complex problems, including numerical boundary layers \cite{Minion2003}.
The OR phenomenon can be seen by studying the RK scheme applied to \eqref{eq:linear_ode_test_problem}. The approximation error at time $t_{n+1}$ reads \cite[Chapter IV.15]{WannerHairer1991}
\begin{equation}
\label{eq:error_formula}
\err^{n+1} = R(\zeta)\,\err^{n}
+ \zeta\vecipower{b}{T}(I-\zeta A)^{-1}\vec{\lte}_s^{\,{n+1}}
+ \lte^{n+1}\;,
\end{equation}
where $R(\zeta)$ is the growth factor, and
\begin{equation*}
\vec{\lte}_s^{\,{n+1}}
=\! \sum_{j\geq 2}\textstyle\frac{\dt^{\,j}}{(j-1)!}\,\sov{j}\phi^{(j)}(t_n)
\;,\quad
\lte^{n+1} =\! \displaystyle\sum_{j\geq 1}\textstyle\frac{\dt^{\,j}}{(j-1)!}
\!\left(\vecipower{b}{T}\vecipower{c}{j-1}-\textstyle\frac{1}{j}\right)\!
\phi^{(j)}(t_n)
\end{equation*}
are the truncation errors incurred at the intermediate stages and at the end of the step, respectively. Here, $\phi^{(j)}$ denotes the $j$-th derivative of the solution, and the vectors
\begin{equation*}
\sov{j} = A\vecipower{c}{j-1} - \tfrac{1}{j}\vecipower{c}{j}\;, \quad j=1,2,\dots
\end{equation*}
we call the \emph{stage order residuals} or \emph{stage order vectors}. The condition $\sov{\eta}=0$ for $0 \le \eta \le j$ appears often in the literature and is also referred to as the simplifying assumption $C(\eta)$ \cite{WannerHairer1991}. In \eqref{eq:error_formula}, the step error $\lte^{n+1}$ is of the formal order (in $\dt$) of the scheme (due to the order conditions). Moreover, the growth factor carries over (more or less, see \cite{DitkowskiGottlieb2017}) the accuracy from one to the next step. Hence, the critical expression for OR is the term involving the stage error $\vec{\lte}_s^{\,{n+1}}$. Specifically, the asymptotic behavior of the expression
\begin{equation}
\label{eq:error_stages}
g^{(j)} = \zeta\vecipower{b}{T}(I-\zeta A)^{-1}\sov{j}
\end{equation}
matters. In the non-stiff limit ($\zeta \ll 1$), a Neumann expansion yields $\zeta(I-\zeta A)^{-1} = \zeta I+\zeta^2 A+\zeta^3 A^2+\dots$, leading to expressions $\vecipower{b}{T}A^\ell\sov{j}$ with $\ell>0$. And in fact the order conditions guarantee that $\vec{b}^T A^\ell \sov{j} = 0$ for $0 \leq \ell + j \leq p -1$ to ensure the formal order of the scheme.

Conversely, in the stiff limit we can treat $\zeta^{-1}$ as the small parameter and expand
$\zeta(I-\zeta A)^{-1} = -A^{-1}(I-\zeta^{-1} A^{-1})^{-1} = -A^{-1}-\zeta^{-1} A^{-2}-\zeta^{-2} A^{-3}-\dots$, leading to expressions $\vecipower{b}{T}A^\ell\sov{j}$ with $\ell<0$.
The order conditions do \emph{not} imply that these quantities vanish, and in general one may observe a reduced rate of convergence.

A key question is therefore whether additional conditions can be imposed on the RK scheme that recover the scheme's order in the stiff regime. A well-known answer to the question is:
\begin{definition}
Let $\hat{p}$ denote the order of the quadrature rule of an RK scheme.
Let $\hat{q}$ denote the largest integer such that
$\sov{j}=0$ for $1\le j\le \hat{q}$. The \emph{stage order} of a RK scheme is $q=\min(\hat{p},\hat{q})$.
\end{definition}
Having stage order $q$ implies that the error decays at an order of (at least) $q$ in the stiff regime (see also \cite{WannerHairer1991}). This work focuses particularly on diagonally-implicit Runge-Kutta (DIRK) schemes, for which $A$ is lower diagonal. A known drawback of DIRK schemes is that they cannot have high stage order:
\begin{theorem}\label{thm:stage_order}
The stage order of an irreducible DIRK scheme is at most 2. The stage order of a DIRK scheme with non-singular $A$ is at most 1.
\end{theorem}
\begin{proof}
Since $\vecc=A\vece$, we have $\sov{2}_1 = a_{11}c_1 - \frac{1}{2}(c_1)^2 = \frac{1}{2}(a_{11})^2$.  Thus if $A$ is non-singular, one has $\sov{2}\neq 0$, so $q \le 1$. Consider now the case that $a_{11} = c_1 = 0$, and suppose that the method has stage order 3.  The conditions $\sov{2}_2 = \sov{3}_2 = 0$ then imply $a_{21} = a_{22} = c_2 = 0$, which would render the scheme reducible. Hence, $q\le 2$.
\qed
\end{proof}
Hence, while DIRK schemes possess an implementation-friendly structure (each stage is a backward-Euler-type solve), their potential to avoid OR by means of high stage order is limited. We therefore move to a weaker condition that can avoid OR in some situations for higher order in the context of DIRK schemes.

\vspace{1.5em}
\section{Weak Stage Order}
\label{sec:weak_stage_order}
To avoid order reduction, the expressions $g^{(j)}$ in \eqref{eq:error_stages} need to vanish in the stiff limit. In line with \cite{RosalesSeiboldShirokoffZhou2017OR}, we define the following criteria:
\begin{definition}[weak stage order]
\label{def:wso}
A RK scheme has weak stage order (WSO) $\tilde{q}$ if there is an $A$-invariant subspace that is orthogonal to $\vec{b}$ and that contains the stage order vectors $\sov{j}$ for $1\le j\le \tilde{q}$.
\end{definition}
\begin{theorem}(WSO is the most general condition that ensures $g^{(j)} = 0$ for all $\zeta > 0$)
Let coefficients $A, b$ be given.  Then $g^{(j)} = 0$ for all $\zeta > 0$
and $1\leq j \leq \tilde{q}$ if and only if the corresponding RK scheme has weak stage order $\tilde{q}$.
\end{theorem}
\begin{proof}
Let $C(G)$ denote the column space of
\begin{equation*}
G := \begin{bmatrix} \sov{1}, A \sov{1},A^2 \sov{1}, \ldots, A^{s-1}\sov{1}, \sov{2}, A\sov{2}, \ldots, A^{s-1}\sov{\tilde{q}} \end{bmatrix}.
\end{equation*}
From the Cayley-Hamilton theorem it follows that WSO $\tilde{q}$ is equivalent to
\begin{equation}
\label{eq:WSODef_2}
\vec{b}^T A^\ell \sov{j} = 0, \quad \quad 0\leq \ell \leq s-1, \ 1 \leq j \leq \tilde{q}\;.
\end{equation}
\noindent\fbox{$\Longrightarrow$}
Because $C(G)$ is $A$-invariant, $C(G)$ is invariant under multiplication
by $(1-\zeta A)^{-1}$, i.e. if $\vec{v} \in C(G)$ then for any  $\zeta > 0$, the product
$(1-\zeta A)^{-1} \vec{v} \in C(G)$.
Since $\vec{b}$ is orthogonal to $C(G)$, we have $g^{(j)}=0$ for all $1\leq j \leq \tilde{q}$.\\[.2em]
\noindent\fbox{$\Longleftarrow$} If $g^{(j)} = 0$, then $\zeta^{-1} g^{(j)} = \vec{b}^T(1 -\zeta A)^{-1} \sov{j} = 0$ for all $\zeta > 0$. Differentiating both sides of this equation $\ell$-times, with respect to $\zeta$, and taking the limit as $\zeta \rightarrow 0^+$, yields the conditions in equation \eqref{eq:WSODef_2}.
\qed
\end{proof}
\begin{definition}[weak stage order eigenvector criterion]
\label{def:wso_ev}
A RK scheme satisfies the WSO eigenvector criterion of order $\tilde{q}_e$ if for each $1\le j\le \tilde{q}_e$, there exists $\mu_j$ such that $A\sov{j} = \mu_j\sov{j}$, and moreover, $\vecipower{b}{T}\sov{j} = 0$.
\end{definition}
The WSO eigenvector criterion of order $\tilde{q}_e$ implies WSO (of at least) $\tilde{q}_e$. For a given scheme, let $p$ denote the classical order, $q$ the stage order, and $\tilde{q}$ the weak stage order. Then we have $\tilde{q} \ge q$ and $p \ge q$. Note however that a method with WSO $\tilde{q}\ge 1$ need not even be consistent; order conditions must be imposed separately.

The WSO eigenvector criterion may serve to avoid OR because it implies that
\begin{equation*}
g^{(j)} = \zeta\vecipower{b}{T}(1-\zeta\mu_j)^{-1}\sov{j}
= \frac{\zeta}{1-\zeta\mu_j}\vecipower{b}{T}\sov{j}\;,
\end{equation*}
i.e., it allows one to ``push'' the stage order residuals past the matrix $(1-\zeta A)^{-1}$, and then use $\vecipower{b}{T}\sov{j} = 0$. Note that the condition $\vecipower{b}{T}\sov{j} = 0$ that is required in Def.~\ref{def:wso_ev} is actually automatically satisfied (due to the order conditions) if $p>\tilde{q}_e$ (or $p\ge \tilde{q}_e$ for stiffly accurate schemes).

It must be stressed that the concept of WSO (both criteria) is based on the linear test equation \eqref{eq:linear_ode_test_problem}, hence it is not clear to what extent WSO will remedy OR for nonlinear problems or problems with time-dependent coefficients. In Section~\ref{sec:numerical_results} we numerically investigate some nonlinear test problems.

Finally, we present a limitation theorem on the WSO eigenvector criterion.
\begin{theorem}
DIRK schemes with invertible $A$ have $\tilde{q}_e\le 3$.
\end{theorem}
\begin{proof}
Because the $\sov{j}$ only depend on $A$, the eigenvector relation in Def.~\ref{def:wso_ev} depends only on $A$, not on $\vec{b}$. With $A$ lower triangular, the first $k$ components of $\sov{j}$ depend only on the upper $k$ rows of $A$; and the same is true for the eigenvector relation as well. Hence, for a scheme to have an $A$ that allows for the WSO eigenvector criterion of order $\tilde{q}_e$, all upper sub-matrices of $A$ must admit the same, too. We can therefore study $A$ row by row.
The first component of $\sov{j}$ equals $(1-\frac{1}{j})a_{11}^j$, which is nonzero for $j>1$. Hence, the first row of the equation $A\sov{j} = \mu_j\sov{j}$ is equivalent to $\mu_j = a_{11}$. With that, we can move to the second row of the equation, which reads
\begin{equation}
\label{eq:weak_stage_order_condition_DIRK_2x2}
(1\!-\!\tfrac{1}{j})a_{11}^j a_{21}
+ (a_{22}\!-\!a_{11})
\prn{a_{11}^{j-1}a_{21} + (a_{21}\!+\!a_{22})^{j-1}a_{22} - \tfrac{1}{j}(a_{21}\!+\!a_{22})^j}
= 0\;.
\end{equation}
To determine the set of solutions $(a_{11},a_{21},a_{22})$ of \eqref{eq:weak_stage_order_condition_DIRK_2x2}, we first observe that \eqref{eq:weak_stage_order_condition_DIRK_2x2} is homogeneous, i.e., if $(a_{11},a_{21},a_{22})$ solves \eqref{eq:weak_stage_order_condition_DIRK_2x2}, then $(\mu a_{11},\mu a_{21},\mu a_{22})$ solves \eqref{eq:weak_stage_order_condition_DIRK_2x2} as well for any $\mu\in\mathbb{R}$. It therefore suffices to consider the solutions of \eqref{eq:weak_stage_order_condition_DIRK_2x2} in the 2D-plane $(\frac{a_{11}}{a_{21}},\frac{a_{22}}{a_{21}})$. Figure~\ref{fig:order_curves_upper_block} shows the resulting solution curves for $j\in\{2,3,4\}$.

\begin{figure}
\begin{minipage}[b]{.46\textwidth}
\includegraphics[width=\textwidth]{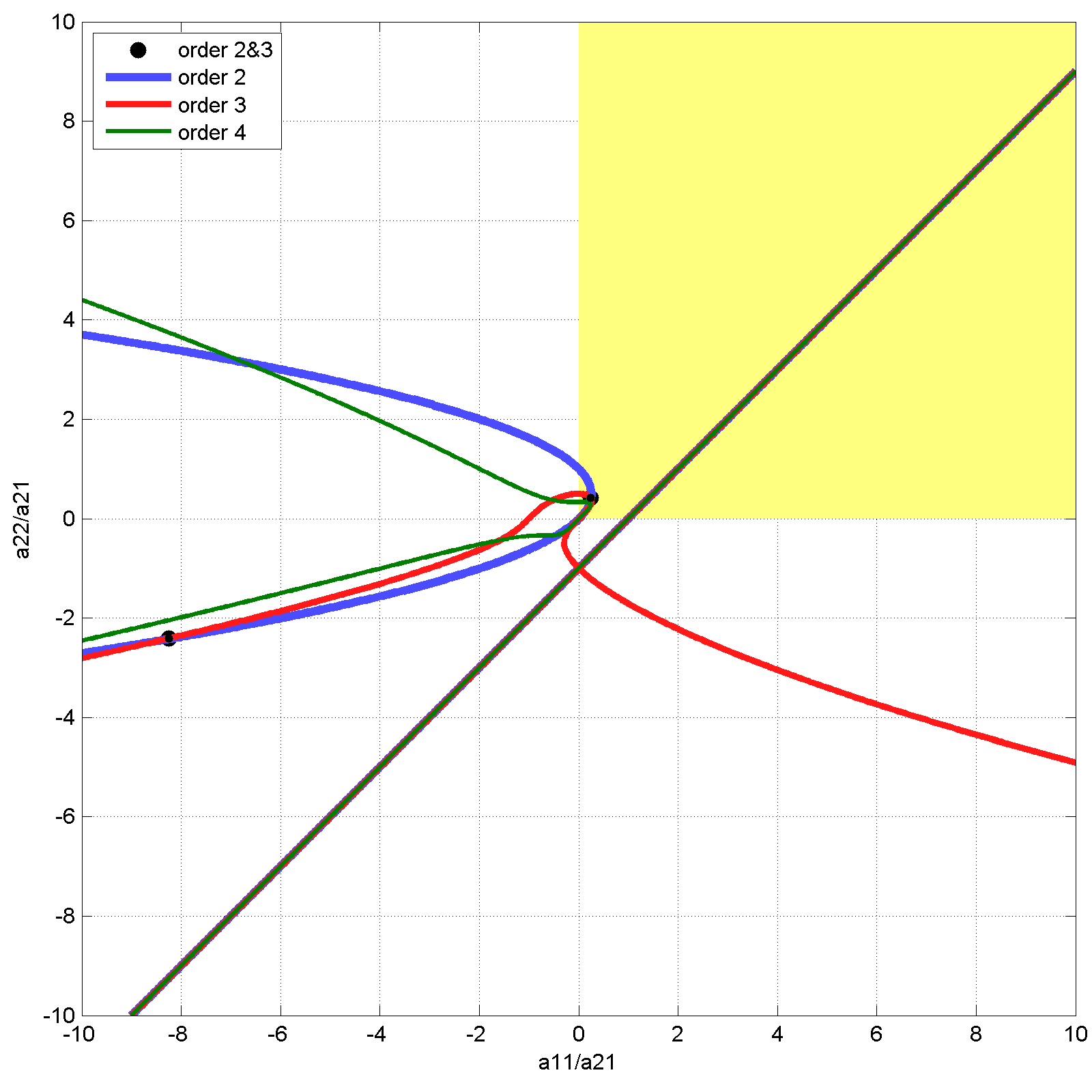}
\end{minipage}
\hfill
\begin{minipage}[b]{.46\textwidth}
\includegraphics[width=\textwidth]{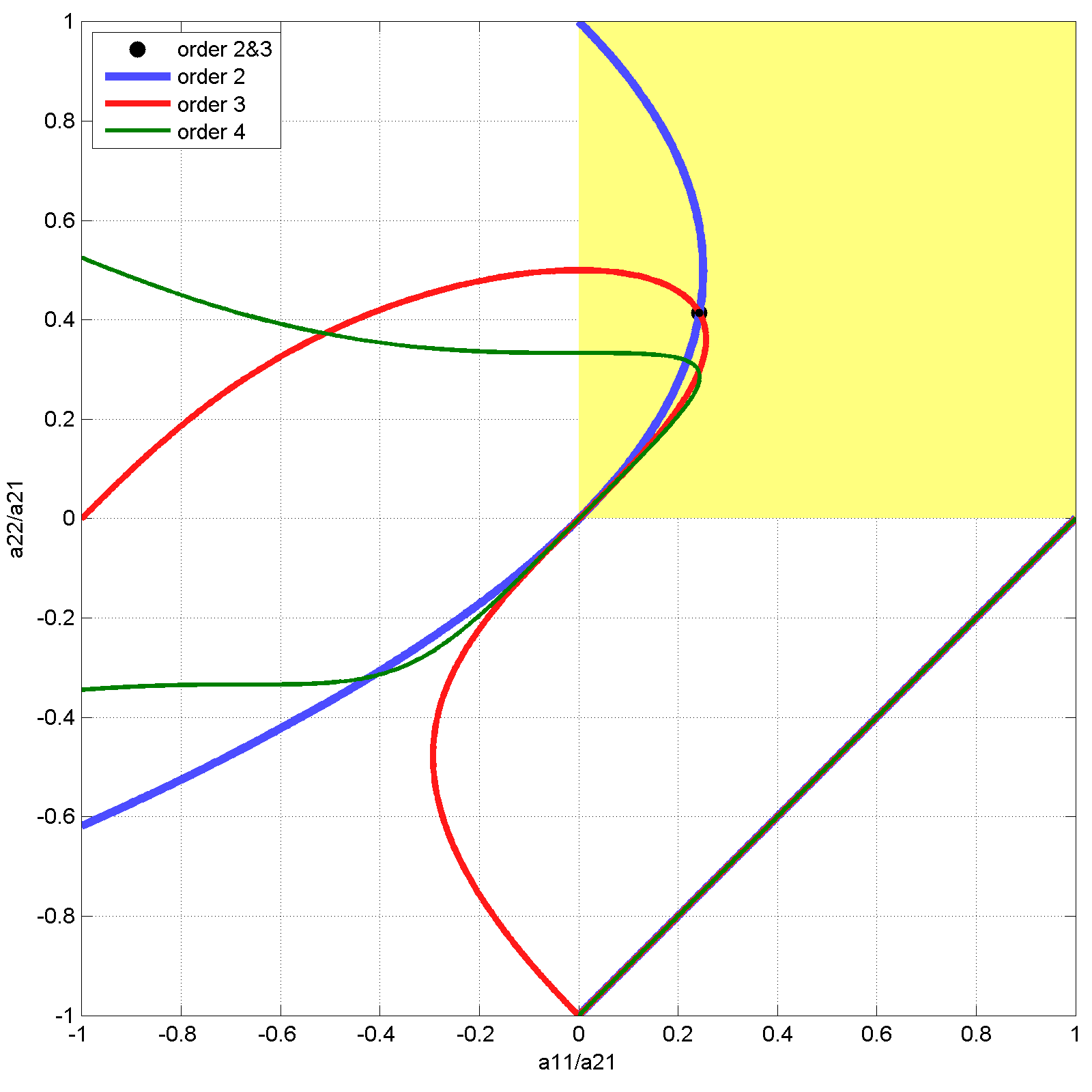}
\end{minipage}
\caption{Curves of WSO orders 2,3, and 4 as functions of the re-scaled parameters $\frac{a_{11}}{a_{21}}$ and $\frac{a_{22}}{a_{21}}$. Left panel:~scale 10; right panel:~scale 1. All orders are satisfied along the line of slope 1 going through (1,0), corresponding to equal-time DIRK schemes. Moreover, there are two further points (other than the origin), where orders 2 and 3 are satisfied. Neither of these two points satisfies order 4.}
\label{fig:order_curves_upper_block}
\end{figure}

One class of solutions lies on the straight line of slope 1 passing through $(1,0)$. Those schemes are \emph{equal-time} methods, i.e., RK schemes that have $\vec{c} = \nu\vec{e}$, where $\nu\in\mathbb{R}$ is a constant. In fact, equal-time schemes satisfy the eigenvector relation for all $j$. However, they are not particularly useful RK methods, because---among other limitations---they are restricted to second order. This follows because the order 1 and 2 conditions require $\vecpower{b}{T}\vec{e} = 1$ and $\vecpower{b}{T}\vec{c} = \frac{1}{2}$. Thus $\nu = \frac{1}{2}$, and $\vecpower{b}{T}\vec{c}^2 = \nu^2 = \frac{1}{4}$, which contradicts the order 3 condition $\vecpower{b}{T}\vec{c}^2 = \frac{1}{3}$. Note that the equal-time scenario also covers the points at infinity in Fig.~\ref{fig:order_curves_upper_block}, i.e., the schemes with $a_{21} = 0$.

Non-equal-time schemes that satisfy \eqref{eq:weak_stage_order_condition_DIRK_2x2} for $j=2$ and $j=3$ are the following two points in the $(\frac{a_{11}}{a_{21}},\frac{a_{22}}{a_{21}})$ plane: $P_1 = (-4+3\sqrt{2},\sqrt{2}-1) = (0.2426,0.4142)$ and $P_2 = (-(\sqrt{2}+1)(\sqrt{2}+2),-(\sqrt{2}+1)) = (-8.2426,-2.4142)$. None of these two points satisfies \eqref{eq:weak_stage_order_condition_DIRK_2x2} for $j=4$ (green curve in Fig.~\ref{fig:order_curves_upper_block}). Therefore $\tilde{q}_e\le 3$.
\qed
\end{proof}

Among the two sets of solutions found in the proof, $P_1$ implies that $a_{11}$, $a_{21}$, and $a_{22}$ all have the same sign, which is a desirable property. In contrast, $P_2$ implies that $a_{21}<0$. Both WSO 3 schemes presented below correspond to the $P_1$ solution.

\vspace{1.5em}
\section{DIRK Schemes With High Weak Stage Order}
\label{sec:schemes}
Imposing the classical order conditions \cite{ButcherBook2008,HairerNorsettWanner1993}, together with the WSO eigenvector relation (Def.~\ref{def:wso_ev}), we determine RK schemes by searching the parameter space of DIRK schemes (with all diagonal entries non-zero). A stiffly accurate structure ($\vecipower{b}{T}$ equals the last row of $A$) is imposed, as is A-stability (verified by evaluating the stability function $R(\zeta)$ along the imaginary axis). Together this implies that the resulting scheme is L-stable; i.e., it ensures that unresolved stiff modes decay \cite{HairerNorsettWanner1993}. The number of stages is chosen so that the constraints admit solutions. The optimization itself is carried out using MATLAB's optimization toolbox, using multiple local optimization algorithms included in the function \texttt{fmincon}. An effort was made to minimize the $L_2$ norm of the local truncation error coefficients. However, in multiple cases the solver exhibited bad convergence properties; so while the schemes below yield reasonable truncation errors, it should not be expected that they are optimal. We find an order 3 scheme with WSO 2 (see also \cite{RosalesSeiboldShirokoffZhou2017OR}),
\begin{equation*}
\begin{array}{l|@{~~~}l@{~}l@{~~~}l@{~~~}l}
0.01900072890 & 0.01900072890 & & & \\
0.78870323114 & 0.40434605601 & \phantom{-}0.38435717512 & & \\
0.41643499339 & 0.06487908412 & -0.16389640295 & 0.51545231222 & \\
1             &	0.02343549374 & -0.41207877888 & 0.96661161281 & 0.42203167233\\
\hline
& 0.02343549374 & -0.41207877888 & 0.96661161281 & 0.42203167233
\end{array}
\end{equation*}
an order 3 scheme with WSO 3,
\begin{equation*}
\begin{array}{l|@{~}l@{~}l@{~}l@{~~~}l}
0.13756543551 &\phantom{-}0.13756543551 & & & \\
0.80179011576 &\phantom{-}0.56695122794 &\phantom{-}0.23483888782 & & \\
2.33179673002 &-1.08354072813 &\phantom{-}2.96618223864 &\phantom{-}0.44915521951 & \\
1             &\phantom{-}0.59761291500 &-0.43420997584 &-0.05305815322 & 0.88965521406 \\
\hline
&\phantom{-}0.59761291500 &-0.43420997584 &-0.05305815322 & 0.88965521406
\end{array}
\end{equation*}
and an order 4 scheme with WSO 3,
\begin{equation*}
\begin{footnotesize}
\begin{array}{l|@{~}l@{~~~}l@{~}l@{~}l@{~~~}l@{~~~}l}
\sm{0.079672377876931} &\sm{\phantom{-}0.079672377876931} &                 \sm{0} &      \sm{\phantom{-}0} &                 \sm{\phantom{-}0} &                 \sm{0}                 & \sm{0} \\
\sm{0.464364648310935} &\sm{\phantom{-}0.328355391763968} & \sm{0.136009256546967} &      \sm{\phantom{-}0} &                 \sm{\phantom{-}0} &                 \sm{0}                 & \sm{0} \\
\sm{1.348559241946724} &\sm{-0.650772774016417} & \sm{1.742859063495349} &\sm{\phantom{-}0.256472952467792} &                 \sm{\phantom{-}0} &                 \sm{0}                 & \sm{0} \\
\sm{1.312664210308764} &\sm{-0.714580550967259} & \sm{1.793745752775934} &\sm{-0.078254785672497} & \sm{\phantom{-}0.311753794172585} & \sm{0}                 & \sm{0} \\
\sm{0.989469293495897} &\sm{-1.120092779092918} & \sm{1.983452339867353} &\sm{\phantom{-}3.117393885836001} &\sm{-3.761930177913743} & \sm{0.770646024799205} &                 \sm{0} \\
\sm{1}                 &\sm{\phantom{-}0.214823667785537} & \sm{0.536367363903245} &\sm{\phantom{-}0.154488125726409} &\sm{-0.217748592703941} &  \sm{0.072226422925896} & \sm{0.239843012362853} \\
\hline
\sm{1}                 &\sm{\phantom{-}0.214823667785537} & \sm{0.536367363903245} &\sm{\phantom{-}0.154488125726409} &\sm{-0.217748592703941} &  \sm{0.072226422925896} & \sm{0.239843012362853}
\end{array}
\end{footnotesize}
\end{equation*}

\vspace{1.5em}
\section{Numerical Results}
\label{sec:numerical_results}
In this section we verify the order of accuracy of the schemes above and demonstrate that WSO remedies order reduction for linear problems. We confirm that WSO $p$ is required for ODEs, and WSO $p-1$ is required for PDE IBVPs. In addition, we study the effect of WSO for two nonlinear problems.

\subsection{Linear ODE test problem}
We consider the linear ODE test problem \eqref{eq:linear_ode_test_problem} with the true solution $\phi(t) = \sin(t + \frac{\pi}{4})$, the stiffness parameter $\lambda = -10^4$, and the initial condition $u(0) = \sin(\frac{\pi}{4})$. The problem is solved using three 3rd order DIRK schemes (with WSO 1, 2, and 3) and two 4th order DIRK schemes (with WSO 1 and 3)\footnote{We do not construct an order 4 scheme with WSO 2, as we see no role for such a method.} up to the final time $T=10$. The convergence results are shown in Fig.~\ref{fig:linear_ode_test}. In the stiff regime where $|\zeta| = |\lambda|\dt \gg 1$, first order convergence is observed for the WSO 1 schemes as expected, the WSO 2 scheme improves the convergence rate to 2, and the WSO 3 schemes exhibit 3rd order convergence. In addition to yielding better convergence orders in the stiff regime, the schemes with higher WSO also turn out to yield substantially smaller error constants in the non-stiff regime ($\Delta t\ll 1/|\lambda|$). For comparison, we also display a DIRK scheme with explicit first stage (EDIRK), that is, $a_{11}=0$, of stage order 2 (see Thm.~\ref{thm:stage_order}). The left panel of Fig.~\ref{fig:linear_ode_test} shows that the WSO 2 scheme exhibits the same convergence behavior as the stage order 2 EDIRK scheme and performs equally well in terms of accuracy.

\begin{figure}[thb]
\centering
\includegraphics[width = 0.45\textwidth]{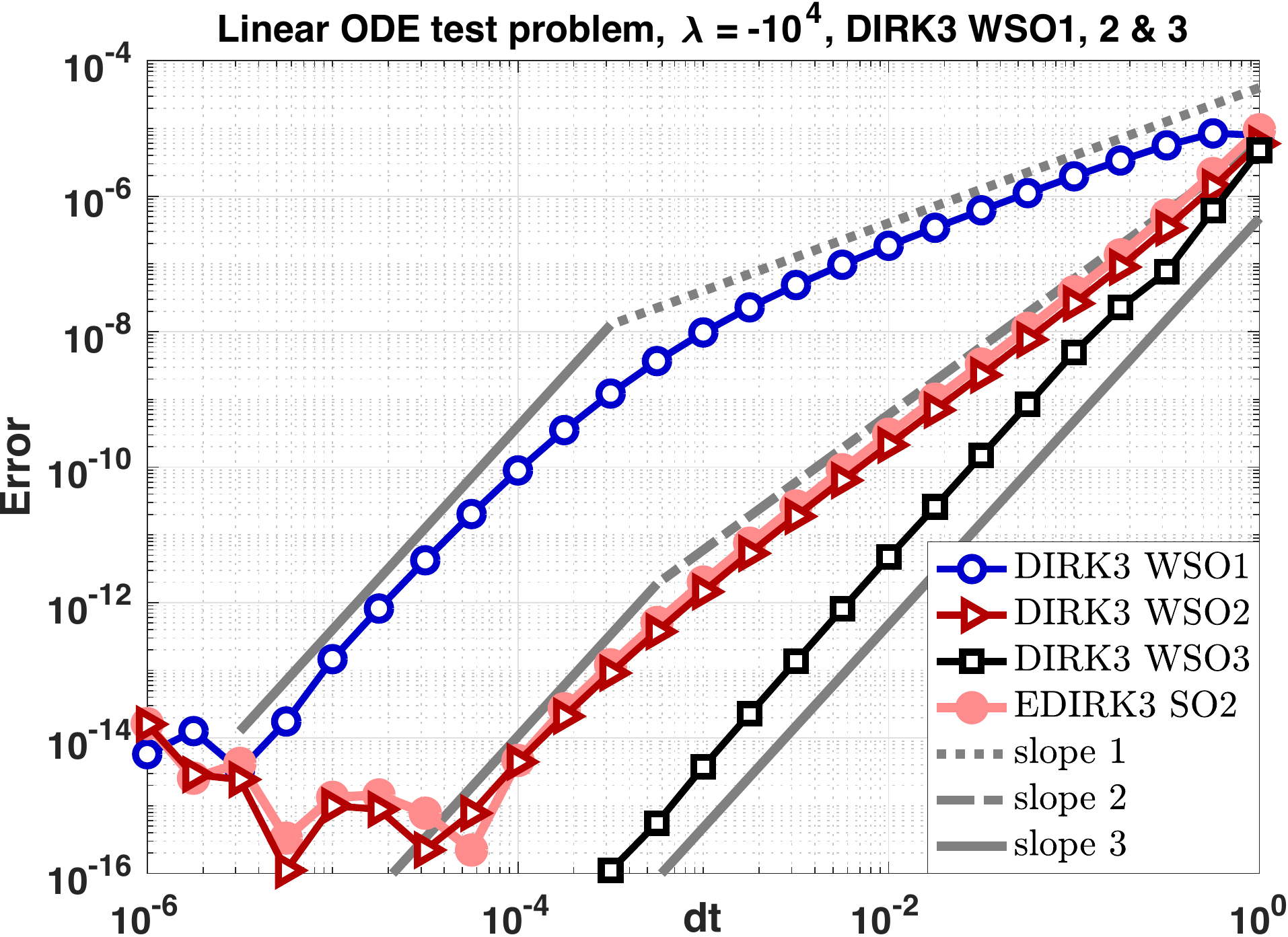}
\includegraphics[width = 0.45\textwidth]{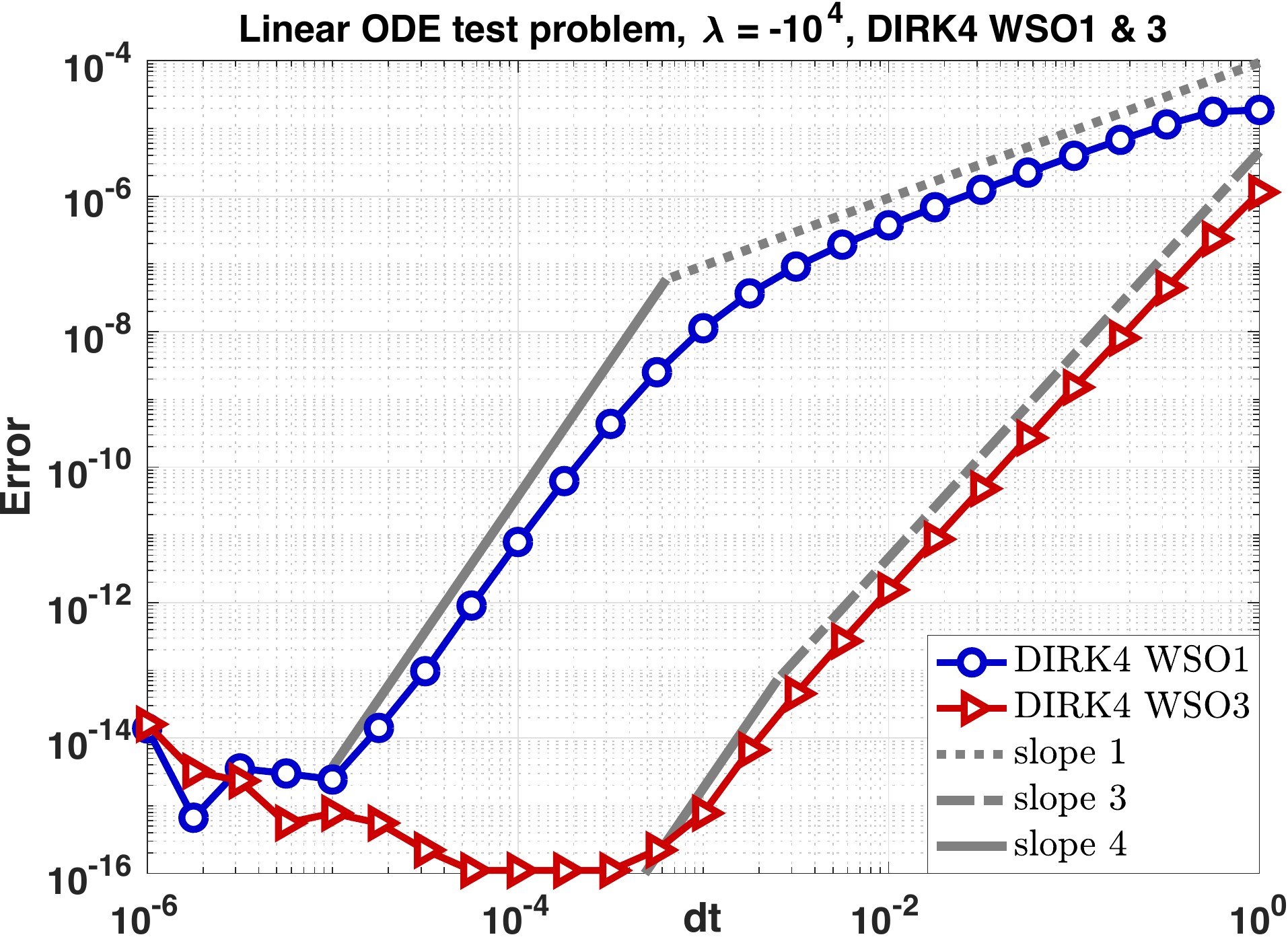}
\caption{Error convergence for linear ODE test problem \eqref{eq:linear_ode_test_problem}.
	Left: 3rd order DIRK schemes with WSO 1 (blue circles), WSO 2 (red triangles), WSO 3 (black squares), and a 3rd order EDIRK scheme with stage order 2 (light red dots).
	Right: 4th order DIRK schemes with WSO 1 (blue circles) and WSO 3 (red triangles).}
\label{fig:linear_ode_test}
\end{figure}

\subsection{Linear PDE test problem: Schr\"odinger equation}
As a linear PDE test problem, we study the dispersive Schr\"odinger equation. The method of manufactured solutions is used, i.e., the forcing, the boundary conditions (b.c.) and initial conditions (i.c.) are selected to generate a desired true solution. The spatial approximation is carried out using 4th order centered differences on a fixed spatial grid of 10000 cells. This renders spatial approximation errors negligible and thus isolates the temporal errors due to DIRK schemes. The errors are measured in the maximum norm in space.

\begin{figure}[thb]
	\centering
	\includegraphics[width = 0.32\textwidth]{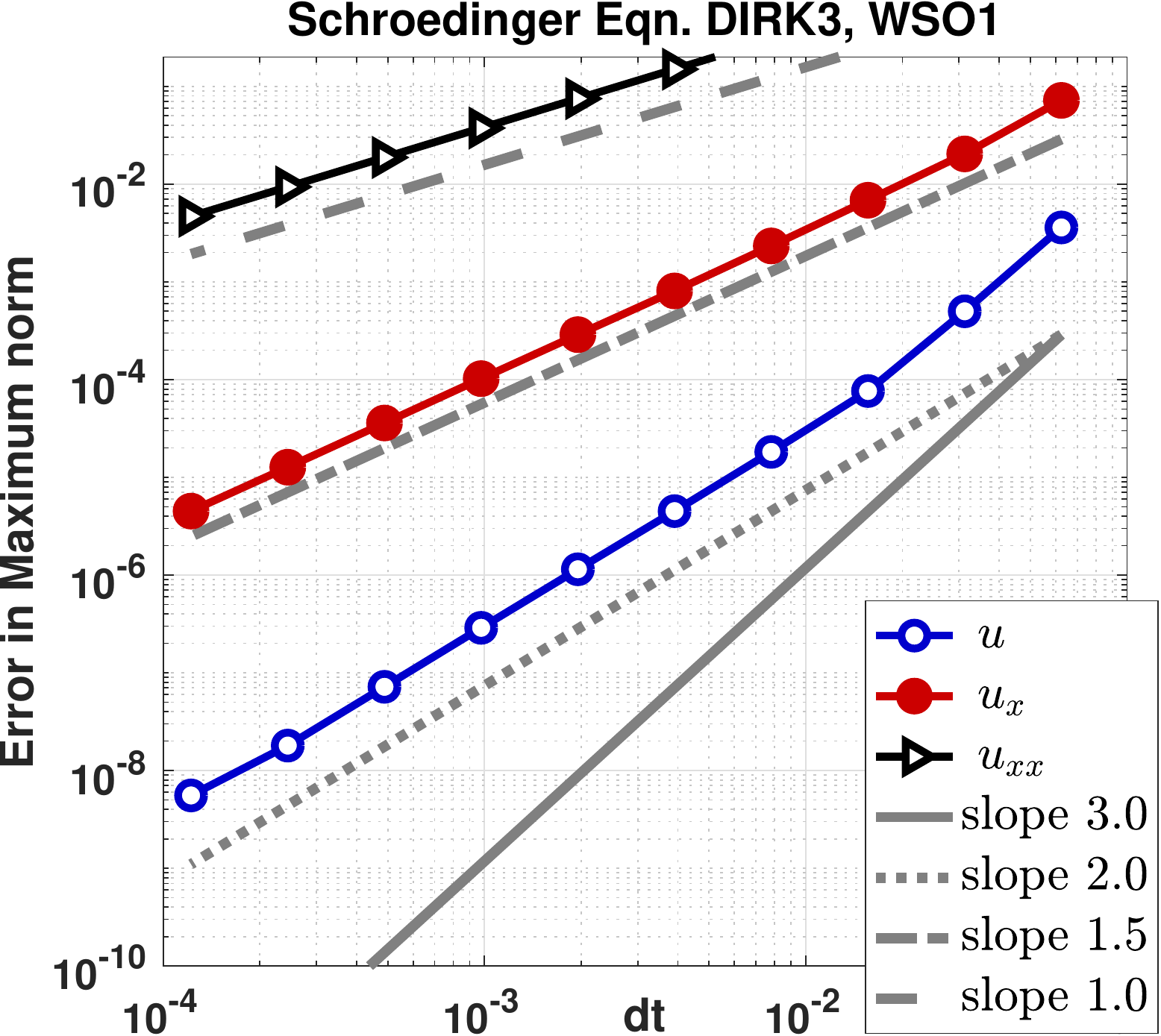}
	\includegraphics[width = 0.32\textwidth]{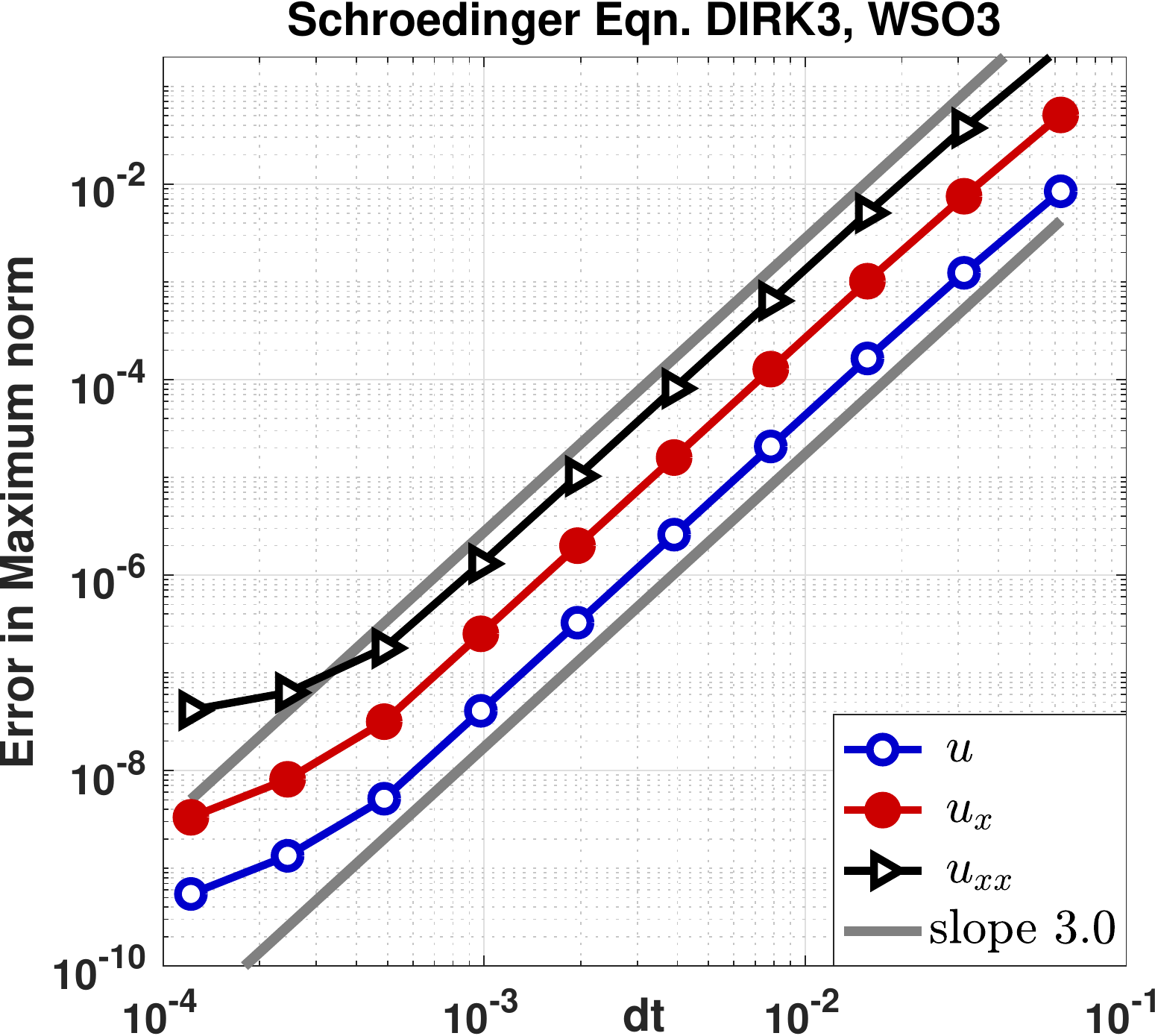}
    \includegraphics[width = 0.32\textwidth]{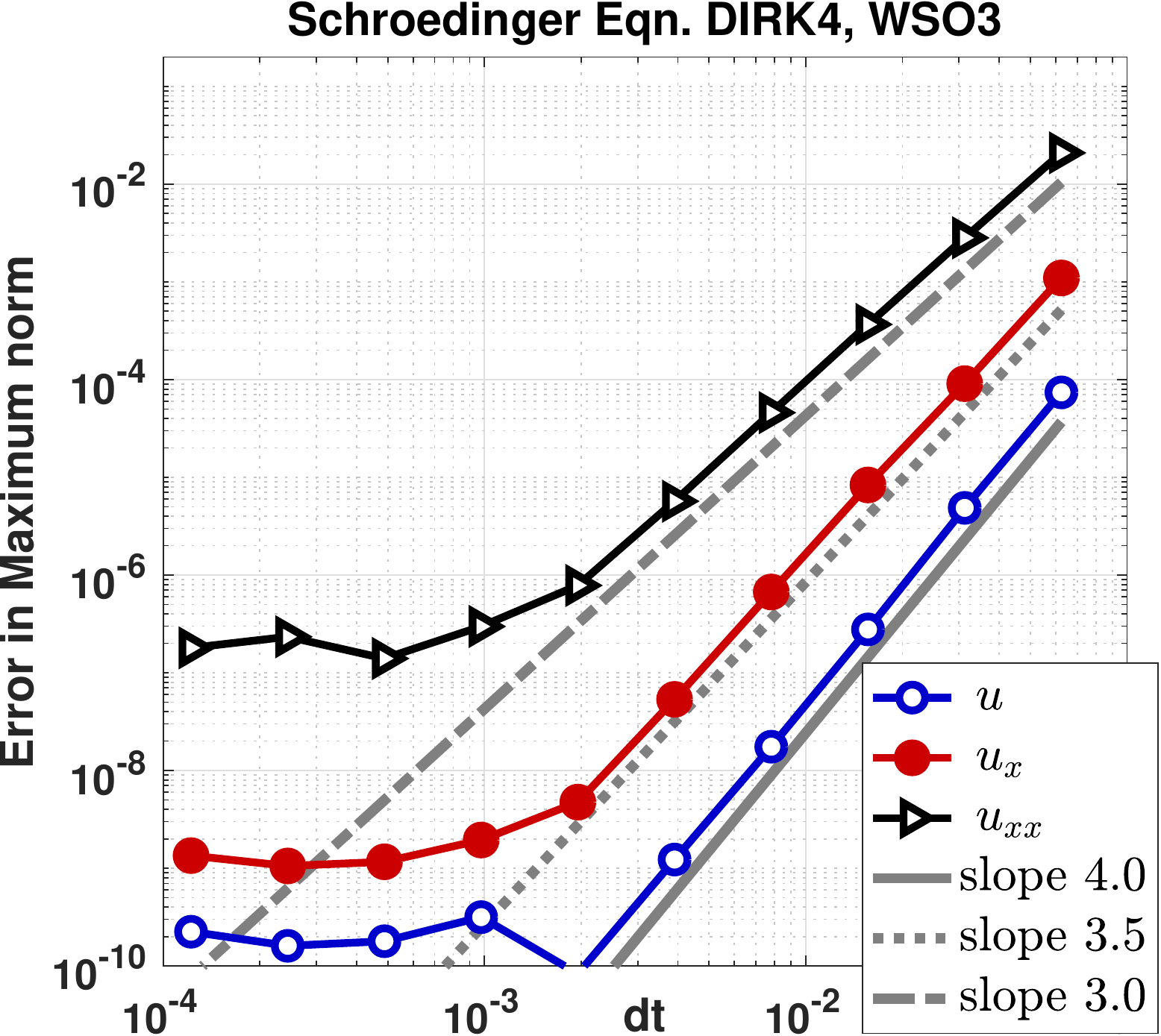}
	\caption{Error convergence for the Schr\"odinger equation using 3rd order DIRK schemes with WSO 1 (left) and WSO 3 (middle), and a 4th order DIRK with WSO 3 (right).}
	\label{fig:schro_eqn_test}
\end{figure}

We consider
\begin{equation}
u_t = \frac{i \omega}{k^2}u_{xx}\;\;\mbox{for}\;\;(x,t)\in(0,1)\times(0,1.2],\quad u = g\;\;\mbox{on}\;\;\{0,1\}\times(0,1.2]\;,
\end{equation}
with the true solution $u(x,t) = e^{i(kx-\omega t)}$, $\omega=2\pi$ and $k=5$.
Figure~\ref{fig:schro_eqn_test} shows the convergence orders of $u$, $u_x$ and $u_{xx}$ for 3rd order DIRK schemes with WSO 1 (left), WSO 3 (middle) and a 4th order DIRK scheme with WSO 3 (right). For IBVPs, spatial boundary layers are produced by RK methods, thus limiting the convergence order in $u$ to $\soq + 1$, with an additional half an order loss per derivative when $\soq<p$ \cite{RosalesSeiboldShirokoffZhou2017OR}. As a result, the 4th order WSO 3 scheme recovers 4th order convergence in $u$ and improves the convergence in $u_x$ and $u_{xx}$. When $\soq=p$, the full convergence order in $u$, $u_x$ and $u_{xx}$ is achieved, as seen in the middle panel in Fig.~\ref{fig:schro_eqn_test}.

\subsection{Nonlinear PDE test problem: Burgers' equation}
This example demonstrates that WSO avoids order reduction for certain nonlinear IBVPs as well. We consider the viscous Burgers' equation with pure Neumann b.c.
\begin{equation}
u_t + u u_x = \nu u_{xx} + f\;\;\mbox{for}\;\;(x,t)\in(0,1)\times(0,1],\quad
u_x= h\;\;\mbox{on}\;\;\{0,1\}\times(0,1]\;.
\end{equation}
Here $\nu = 0.1$ and $u(x,t) = \cos(2+10t)\sin(0.2+20x)$. The nonlinear implicit equations arising at each time step are solved using a standard Newton iteration. The choice of Neumann b.c.\ distinguishes this example from the one given in \cite{RosalesSeiboldShirokoffZhou2017OR}. With Neumann b.c., the convergence order in $u$ is limited to $\soq+1.5$ (half an order better than with Dirichlet b.c.). Figure~\ref{fig:burgers_test} shows that order reduction arises with the stage order 1 scheme, and that the WSO2 scheme recovers 3rd order convergence for $u$ and $u_x$, and the 3rd order WSO 3 scheme yields 3rd order convergence for $u$, $u_x$ and $u_{xx}$.

\begin{figure}[thb]
	\centering
	\includegraphics[width = 0.32\textwidth]{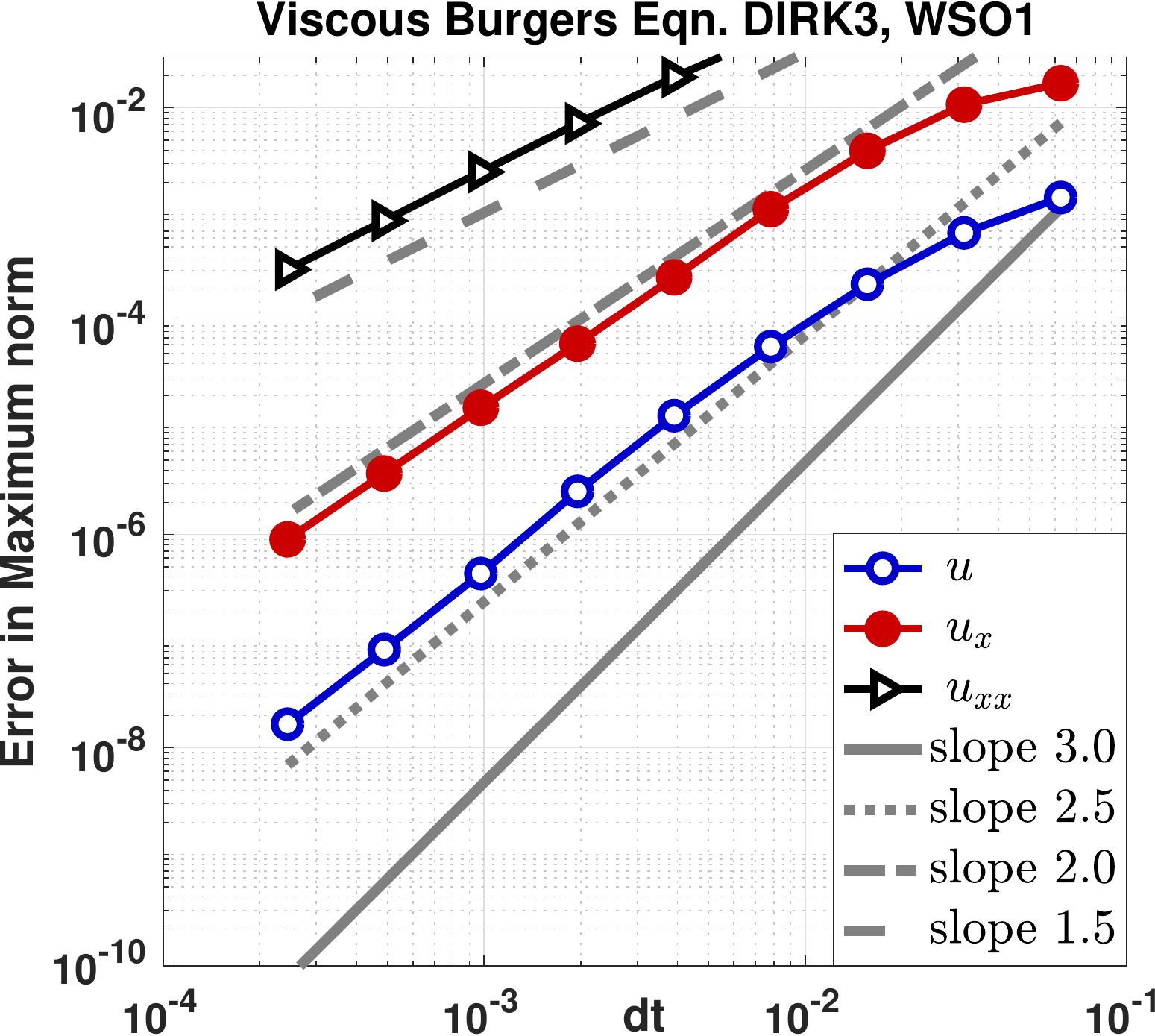}
	\includegraphics[width = 0.32\textwidth]{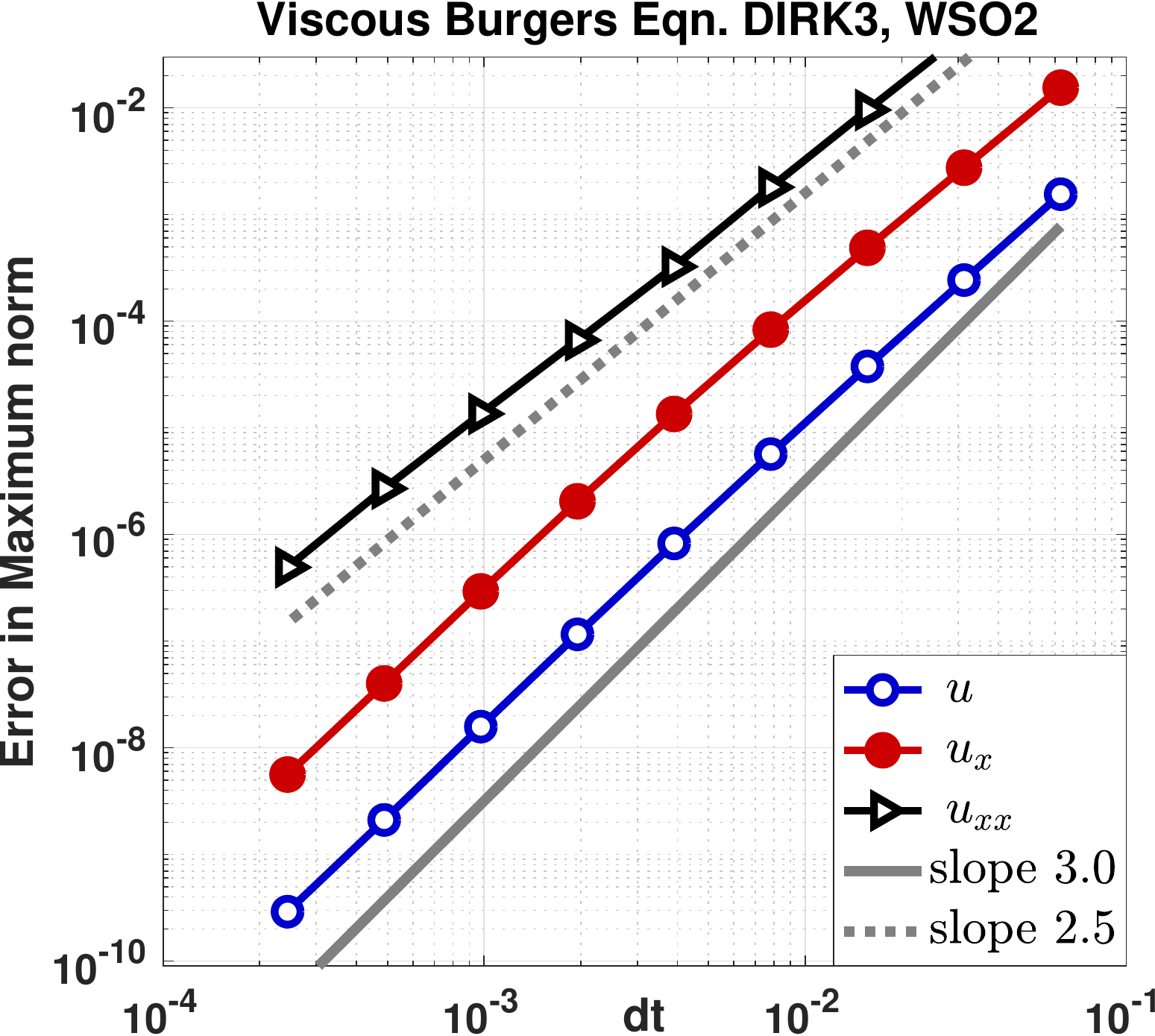}
    \includegraphics[width = 0.32\textwidth]{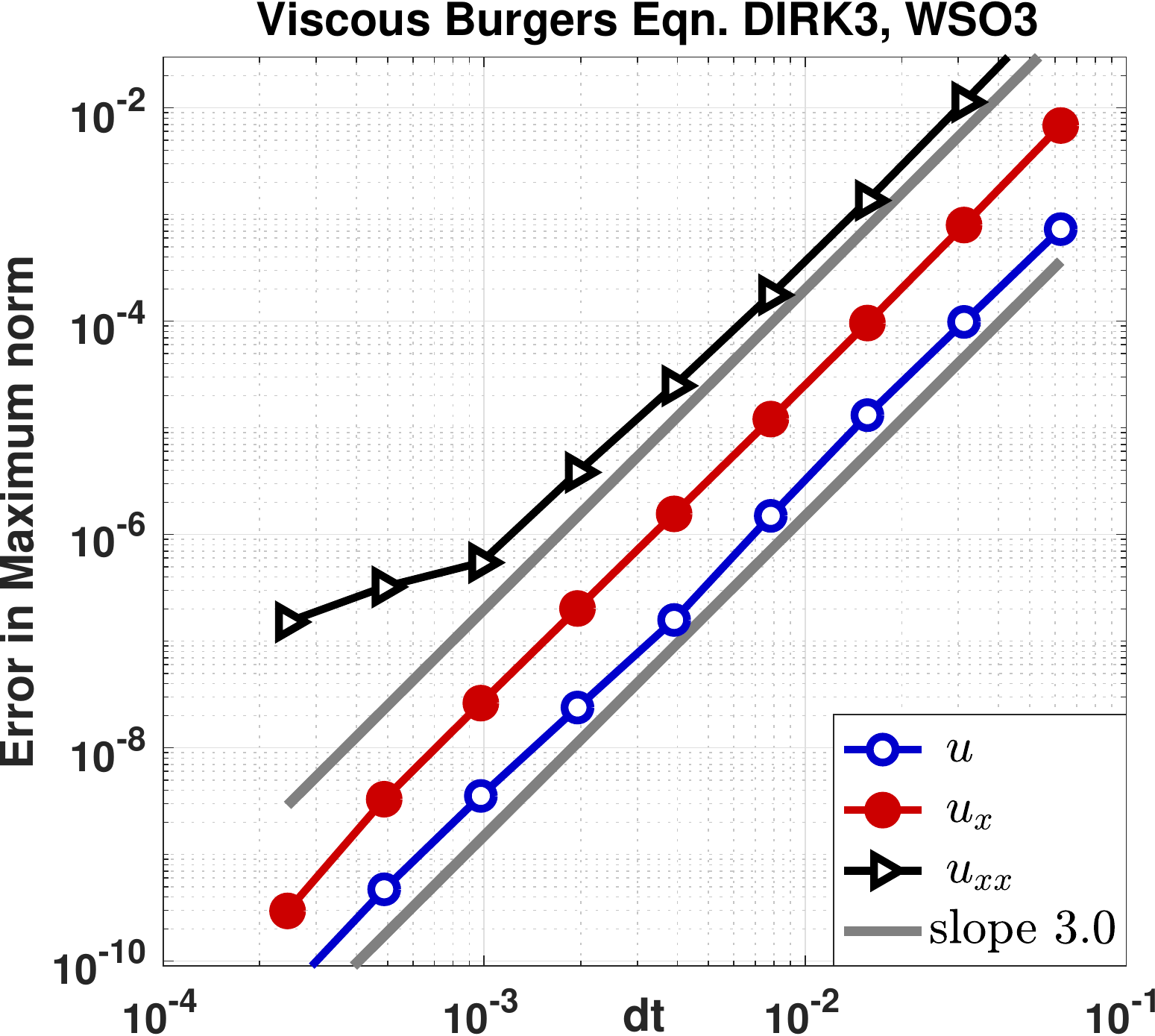}
	\caption{Error convergence for the viscous Burgers' equation using 3rd order DIRK schemes with WSO 1 (left), WSO 2 (middle) and WSO3 (right).}
	\label{fig:burgers_test}
\end{figure}

\subsection{Stiff nonlinear ODE: Van der Pol oscillator}
This example illustrates that DIRK schemes with high WSO may not remove order reduction for all types of nonlinear problems. Consider the Van der Pol oscillator
\begin{equation}
x' = y\;\;\;\text{and}\;\;\; y' = \mu(1-x^2)y-x\;,
\end{equation}
with i.c.\ $(x(0),y(0))=(2,0)$, stiffness parameter $\mu = 500$, and final time $T = 10$. The nonlinear system at each time step is solved via MATLAB's built-in nonlinear system solver. The ``exact'' solution is computed using explicit RK4 with a time step $\dt = 10^{-6}$. In this case, the presented DIRK schemes with high WSO do not improve the convergence rates in the stiff regime and they perform worse than the WSO 1 scheme in terms of accuracy (see Fig.~\ref{fig:vdp_test}). On the other hand, an EDIRK with stage order 2 improves the rate of convergence in the stiff regime (see right panel in Fig.~\ref{fig:vdp_test}). However, it does so, interestingly, by yielding larger errors for large time steps.

\begin{figure}[thb]
	\centering
	\includegraphics[width = 0.45\textwidth]{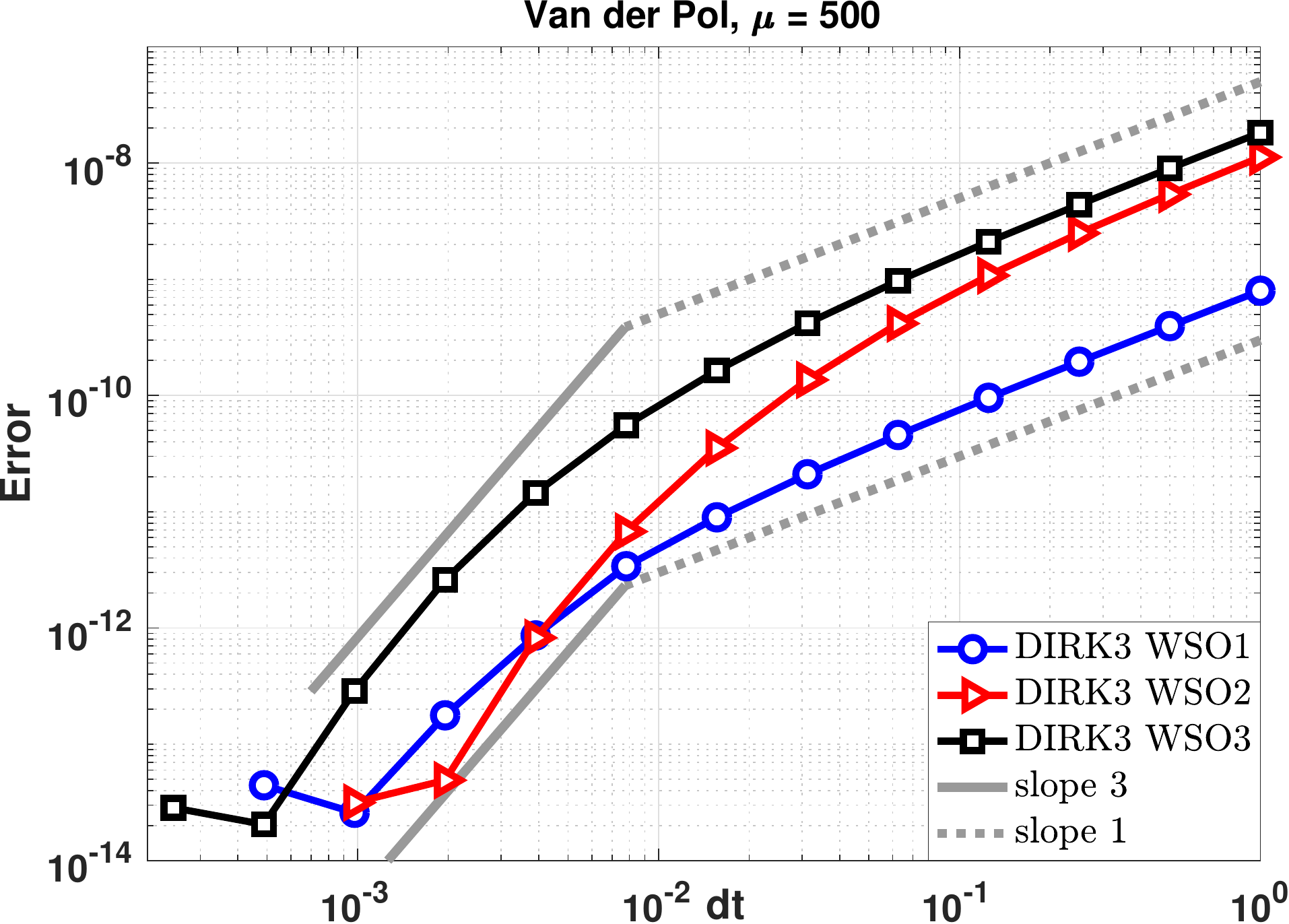}
	\includegraphics[width = 0.45\textwidth]{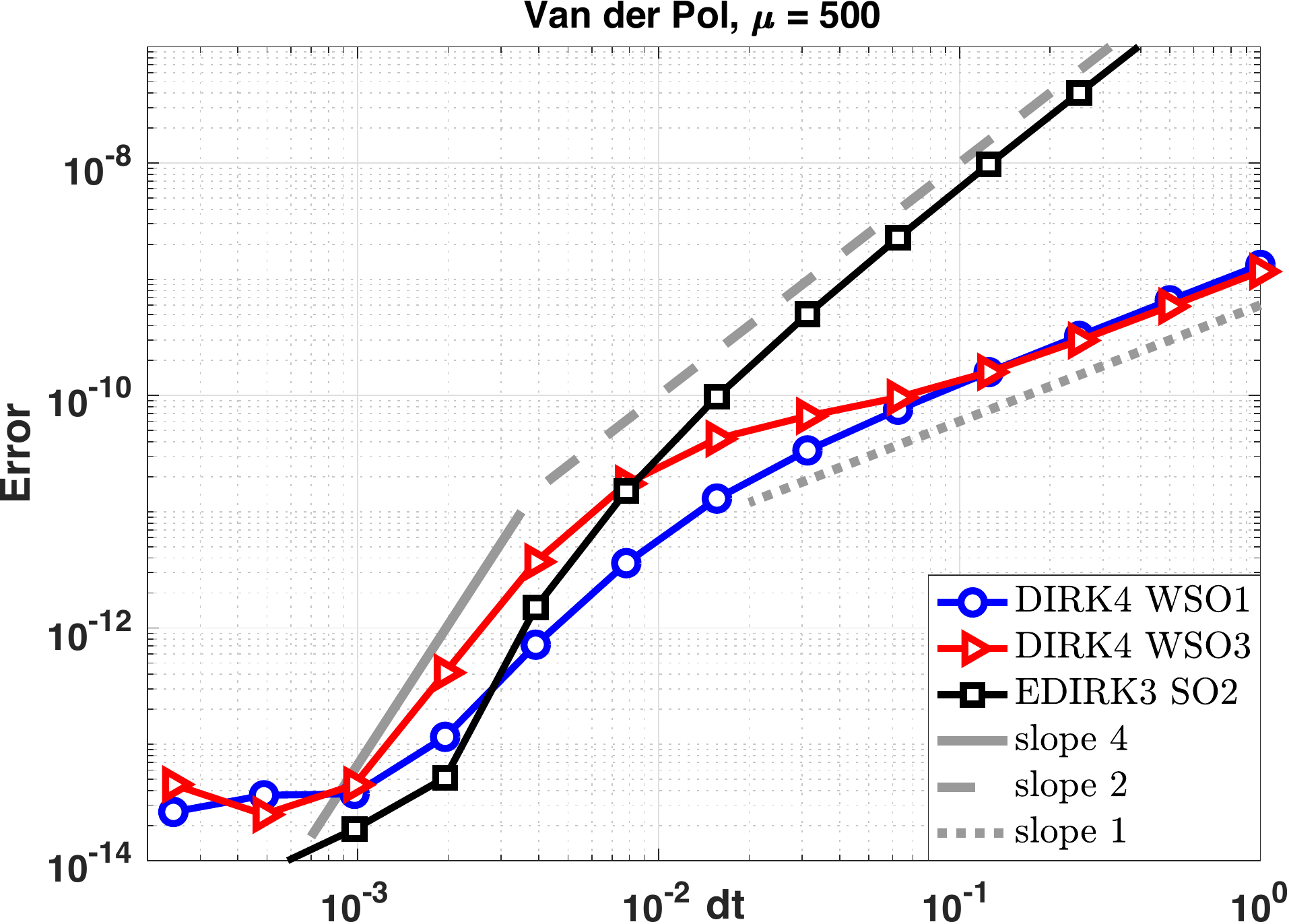}
	\caption{Error convergence for Van der Pol's equation.
		Left: 3rd order DIRK schemes with WSO 1 (blue circles), WSO 2 (red triangles) and WSO 3 (black squares).
		Right: 4th order DIRK schemes with WSO 1 (blue circles) and WSO 3 (red triangles), and a 3rd order EDIRK scheme with stage order 2 (black squares).}
	\label{fig:vdp_test}
\end{figure}

\vspace{1.5em}
\section{Conclusions and Outlook}
This study demonstrates that it is possible to overcome order reduction (OR) for certain classes of problems in the context of DIRK schemes, even though these are limited to low stage order. A specific \emph{weak stage order} (WSO) ``eigenvector'' criterion has been presented, analyzed, and applied to determine DIRK schemes with WSO up to 3. The numerical results confirm that the schemes avoid OR for linear problems and for some nonlinear problems in which the mechanism for order reduction is linear (i.e., boundary conditions). The key limitation found herein is that the eigenvector criterion cannot go beyond WSO 3 for DIRK schemes. Hence, a key question of future research is how high WSO is admitted by the general criterion in Def.~\ref{def:wso}. Another important future research task is to devise further DIRK schemes that are truly optimized in terms of truncation error coefficients or other criteria.

\vspace{1.5em}
\section*{Acknowledgment}
This work was supported by the National Science Foundation via grants DMS--1719640 (BS\&{}DZ), DMS--1719693 (DS), DMS--2012271 (BS), DMS--2012268 (DS); and the Simons Foundation (\#359610) (DS).

\vspace{1.5em}
\bibliographystyle{plain}
\bibliography{references_complete,references_extra}

\vspace{2.5em}
\end{document}